\documentclass[11pt]{amsart}
\usepackage{amsfonts}
\usepackage{amssymb}
\usepackage{graphicx}

\def\a{\alpha}       \def\b{\beta}        

     \def\l{\lambda}       
               \def\z{\zeta}

\def\D{{\mathbb D}}  
\def\C{{\mathbb C}}

\def\({\left(}       \def\){\right)}

\newtheorem{prop}{\sc Proposition}
\newtheorem{lem}{\sc Lemma}
\newtheorem{thm}{\sc Theorem}
\newtheorem{cor}{\sc Corollary}

\newtheorem{other}{\sc Theorem}              
\newenvironment{pf}{\noindent{\textit{Proof. }}}{$\Box$ }

\begin{document}
\title[Affine and linear invariant families]
{Affine and linear invariant families of harmonic mappings}

\author[M. Chuaqui]{Martin Chuaqui}
\address{Facultad de Matem\'aticas, Pontificia Universidad Cat\'olica de Chile, Casilla 306, Santiago, Chile.} \email{mchuaqui@mat.puc.cl}

\author[R. Hern\'andez]{Rodrigo Hern\'andez}
\address{Facultad de Ingenier\'{\i}a y Ciencias, Universidad Adolfo Ib\'a\~nez, Av. Padre Hurtado 750, Vi\~na del Mar, Chile.} \email{rodrigo.hernandez@uai.cl}

\author[M. J. Mart\'{\i}n]{Mar\'{\i}a J. Mart\'{\i}n}
\address{Department of Physics and Mathematics, University of Eastern Finland, P.O. Box 111, FI-80101 Joensuu, Finland.} \email{maria.martin@uef.fi}

\subjclass[2010]{30C55}
\keywords{Univalent function, harmonic mapping, Schwarzian derivative, affine and linear invariant family, order}
\date{\today}
\thanks{The authors are partially supported by grants Fondecyt $1110160$ and $1110321$, Chile. The second author also thankfully acknowledges partial
support from Faculty of Forestry and Sciences UEF, Finland (930349). The third author is supported by
Academy of Finland grant 268009 and by Spanish MINECO Research Project MTM2012-37436-C02-02.}

\begin{abstract} We study the order of affine and linear invariant families of planar harmonic mappings in the unit disk
and determine the order of the family of mappings with bounded Schwarzian norm. The result
shows that finding the order of the class $\mathcal{S}_H$ of univalent harmonic mappings can be formulated as a question
about Schwarzian norm and, in particular, our result shows consistency between the conjectured order of $\mathcal{S}_H$  and
the Schwarzian norm of the harmonic Koebe function.
\end{abstract}
\maketitle

\section*{Introduction}

The purpose of this paper is to study certain affine and linear invariant families of planar harmonic
mappings defined in the unit disk $\D$, with a special interest in the family of mappings with bounded
Schwarzian norm. The fundamental aspects of linear invariant families of harmonic mappings were studied
in \cite{S-S}, while linear invariant families of holomorphic mappings were introduced by Pommerenke in \cite{POM-I}.
Several important properties of such families of either holomorphic or harmonic mappings depend
on its order, namely the optimal bound for the second Taylor coefficient of the holomorphic part of the mappings
considered. The order of the class of holomorphic mappings with bounded Schwarzian norm can be determined by means of
a variational method that gives a relation for the second and third order coefficients of an extremal mapping \cite{POM-I}.
Our motivation in this paper stems from the still unresolved problem of determining the order of the family $\mathcal{S}_H$
of normalized univalent harmonic mappings. In this direction, we are able to apply the variational approach that leads
to the Marty relations to determine the order of the family of harmonic mappings with a given bound for
the Schwarzian derivative. This seems relevant because the order of the class $\mathcal{S}$ of normalized univalent
holomorphic mappings can be derived from the above coefficient relation and the well-known Schwarzian bound for the class.
Our result shows consistency between the conjectured values for order in $\mathcal{S}_H$ and in
$\mathcal{S}_H^0$ and the Schwarzian norm of the harmonic Koebe function, a natural candidate for maximizing the Schwarzian norm in $\mathcal{S}_H^0$.
The hyperbolic norm of the dilatations of the harmonic mappings enter in our analysis in an unexpected way, and
turn out to be intimately related to the order of the family.
The construction of an extremal mapping for our main result is not elementary and depends on a subtle interplay between the second coefficient and the hyperbolic norm. In another direction, we also establish the
sharp bound for the Schwarzian norm of certain important families of harmonic mappings for which the order
was already known.

\par

The \emph{Schwarzian derivative} of a locally univalent analytic function on a domain in the complex plane is
\[
Sf=\left(\frac{f''}{f'}\right)'-\frac 12\,\left(\frac{f''}{f'}\right)^2 \,.
\]
The role of $Sf$ in the study of univalence, distortion, and extensions of $f$ has been developed extensively in the literature (see, \emph{e.g.},
\cite{CH-O, Eps, Kraus, N}).
Two of the main properties of the Schwarzian derivative are the following:
\begin{itemize}
\item[i)] $Sf=0$ if and only if $f$ is a M\"{o}bius transformation.
\item[ii)] Whenever the composition $f\circ g$ is well-defined,  the \emph{chain rule} holds:
$$S(f\circ g)=(Sf\circ g)\cdot (g')^2+Sg\,.$$
\end{itemize}
In the case when $f$ is locally univalent in the unit disk $\D$, the \emph{Schwarzian norm}
$$\|Sf\|=\sup_{|z|<1} |Sf(z)|\cdot(1-|z|^2)^2$$ turns out to be invariant
under post-compositions with automorphisms $\sigma$ of the disk. In other words, for any such functions $f$ and $\sigma$,
$$\|S(f\circ\sigma)\|= \|Sf\|\,.$$
\par
In \cite{POM-I, POM-II} Pommerenke studied and carried through a detailed analysis of the so-called \emph{linear invariant
families\,}; that is, fa\-mi\-lies of locally univalent holomorphic functions $f$ in the unit disk normalized by the conditions $f(0)=1-f^\prime(0)=0$ and
which are closed under the  transformation
$$F_{\zeta}(z)=\frac{f\left(\displaystyle\frac{\zeta+z}{1+\overline{\zeta}z}\right)
-f(\zeta)}{(1-|\zeta|^2)f^\prime(\zeta)}\,,\quad \zeta\in\D.$$
Several important properties, such as growth, covering, and distortion are determined by the \emph{order} of a linear invariant family $\mathcal{F}$ defined by
$$\alpha(\mathcal F)=\sup_{f\in\mathcal F}a_2(f)=\frac 12 \sup_{f\in\mathcal F}|f''(0)|\,.$$
 For example, the order of the important class ${\mathcal S}$ of normalized univalent mappings in $\D$ is 2.
 We refer the reader to the books \cite{Dur-Univ} or \cite{P} for more details related to the class $\mathcal S$.
\par\smallskip
In  \cite{POM-I}, Pommerenke proves the following theorem regarding the linear invariant family $\mathcal{H}_{\lambda}$
of normalized locally univalent analytic functions $f$ in the unit disk with $\|Sf\|\leq \lambda$.

\begin{other}\label{thm-Pom}
The order of the family $\mathcal{H}_{\lambda}$ is given by
$$\alpha(\mathcal{H}_{\lambda})=\sqrt{1+\frac{\lambda}{2}}\, .$$
\end{other}
It is a straightforward calculation to show that given $\lambda\geq 0$, the function
\begin{equation}\label{eq-sharp-functions}
\varphi_a(z)=\frac{1}{2a}\left[\left(\frac{1+z}{1-z}\right)^a-1\right]\,, \quad |z| < 1\,,\quad a=\sqrt{\frac \lambda 2+1}\,,
\end{equation} belongs to $\mathcal{H}_{\lambda}$ and satisfies
\[
\frac 12 |\varphi_a''(0)| = \sqrt{1+\frac{\lambda}{2}}\,.
\]
\par\smallskip
Whenever $f\in \mathcal{S}$, its Schwarzian norm is bounded by $6$. Therefore, $\mathcal{S}\subset\mathcal{H}_6$ so that by Theorem~\ref{thm-Pom} we get $|a_2|\leq 2$ for all $f\in\mathcal{S}$, although the bound will hold also for the non-univalent mappings in $\mathcal{H}_6$.
\par\medskip

A planar \emph{harmonic mapping} in a domain $\Omega\subset\C$ is a complex-valued function
$w=f(z)=u(z)+iv(z)$, $z=x+iy$, which is harmonic, that is, $\Delta u=\Delta v=0$.
When $\Omega$ is simply connected, the mapping $f$ has a canonical
decomposition $f=h+\overline{g}$, where $h$ and $g$ are analytic in
$\Omega$. As is usual, we call $h$ the \emph{analytic part of $f$} and $g$ the \emph{co-analytic part of $f$}. The harmonic mapping $f$ is analytic if and only if $g$ is constant.
Lewy \cite{L} proved  that a harmonic mapping  is locally univalent in a domain $\Omega$
if and only if its Jacobian does not vanish. In terms of the canonical decomposition $f=h+\overline{g}$, the Jacobian is given by $|h'|^2-|g'|^2$, and thus, a locally univalent harmonic mapping in a
simply connected domain $\Omega$ will be \emph{sense-preserving} or \emph{sense-reversing} according to whether $|h'|>|g'|$ or
$|g'|>|h'|$ in $\Omega$. Note that $f=h+\overline g$ is sense-preserving in $\Omega$ if and only if $h^\prime$ does not vanish and the (second complex) \emph{dilatation} $\omega=g'/h'$ has the property that $|\omega|<1$ in $\Omega$.
We refer the reader to the book by Duren \cite{Dur-Harm} for an excellent exposition of
harmonic mappings.
\par
Let $\mathcal{F}$ be a family of sense-preserving harmonic mappings $f=h+\overline{g}$ in $\D$, normalized with $h(0)=g(0)=0$ and $h'(0)=1$. The family is said to be {\it affine and linear invariant} ($\mathcal {AL}$ {\it family}) if it
closed under the two operations of \emph{Koebe transform} and \emph{affine change}:
\begin{equation}\label{eq-koebetransform}
K_{\zeta}(f)(z)=\frac{f\left(\displaystyle\frac{z+\zeta}{1+\overline{\zeta} z}\right)-f(\zeta)}{(1-|\zeta|^2)h'(\zeta)}\,,\quad |\zeta|<1\,,
\end{equation}
and
\begin{equation}\label{eq-affinetransform}
A_\varepsilon(f)(z)=\frac{f(z)-\overline{\varepsilon f(z)}}{1-\overline\varepsilon g'(0)}\,,\quad |\varepsilon|<1\,.
\end{equation}
Sheil-Small \cite{S-S} offers an in depth study of affine and linear invariant families $\mathcal{F}$ of harmonic mappings in $\D$. The \emph{order} of the $\mathcal{AL}$ family, given by
$$ \alpha(\mathcal{F})=\sup_{f\in\mathcal F}|a_2(f)|=\frac 12 \sup_{f\in\mathcal F}|h''(0)| \, ,$$
plays once more a special role in the analysis.
\par
A  special example of affine and linear invariant family is the class $\mathcal{S}_H$ of (normalized) sense-preserving harmonic mappings which are univalent in the unit disk. As it is usual, we use $\mathcal{S}_H^0$ to denote the family of functions $f=h+\overline g\in \mathcal{S}_H$ with $g'(0)=0$. It is conjectured that the second Taylor coefficient of the analytic part $h$ of any function in $\mathcal{S}_H^0$ is bounded by $5/2$. If this conjecture were true, we would obtain that the order of $\mathcal{S}_H$ is equal to $3$. The analytic part of the so-called harmonic Koebe function $K\in \mathcal{S}_H^0$ (introduced by Clunie and Sheil-Small in \cite{CSS}) has second coefficient equal to $5/2$.
\par\smallskip
In \cite{HM2}, the authors introduce a definition for the Schwarzian derivative $S_f$ of locally univalent harmonic mappings, which serves as a complement to the definition found in \cite{CH-D-O}. The requirement of the latter that the dilatation be a square has been
replaced in the former by the local univalence. In both cases a chain rule is in order, which  shows that for mappings in $\D$ the norm $||S_f||$ defined as before is invariant under automorphisms of the disk.  But only
the Schwarzian $S_f$ introduced in \cite{HM2} is invariant under affine changes $af+b\overline{f}$, $|a|\neq |b|$. As a result,  the family $\mathcal{F}_{\lambda}$ of sense-preserving harmonic mappings
$f=h+\overline{g}$ in $\D$, with $h(0)=g(0)=0, h'(0)=1$ and $||S_f||\leq \lambda$, is affine and linear invariant. We let  $\mathcal F_\lambda ^0=\{f\in \mathcal F_\lambda\colon g^\prime(0)=0\}$.
\par\smallskip

The hyperbolic norm of the dilatation $\omega$ of a sense-preserving harmonic mapping in $\D$ is given by
\[
\|\omega^*\|=\sup_{z\in\D} \frac{|\omega'(z)|\cdot(1-|z|^2)}{1-|\omega(z)|^2}\, ,
\]
and will play a distinguished role in our analysis. Observe that $\|\omega^*\|\leq 1$ by Schwarz's lemma. As we will show, only for
$\lambda\geq 3/2$ will there exist a mapping in $f\in \mathcal F_\lambda$ with a dilatation of hyperbolic norm equal to 1. Moreover,
for $\lambda<3/2$ the supremum of all hyperbolic norms of mappings in $\mathcal F_\lambda$ will be strictly less than 1.

Let us denote by $\mathcal A_\lambda^0$ (resp. $\mathcal A_\lambda$) the set of \emph{admissible dilatations} of functions $f\in \mathcal F_\lambda ^0$ (resp. $\mathcal F_\lambda $); i.e., $\omega\in \mathcal A_\lambda^0$ (or $\mathcal A_\lambda$) if there exists a harmonic mapping $f=h+\overline g\in \mathcal F_\lambda ^0$ ($\mathcal F_\lambda$) with dilatation $\omega$. The main purpose of this article is to show the following generalization of  Theorem~\ref{thm-Pom}.
\begin{thm}\label{thm-main}
The order of $\mathcal F_\lambda$  is given by
\begin{eqnarray}\label{eq-thm-main}
\alpha\left( \mathcal F_\lambda\right) &=& \sqrt{\frac{\lambda}{2}+1+\frac 12\sup_{f\in\mathcal F_{\lambda}^0}|g''(0)|^2}+\frac 12 \sup_{f\in\mathcal F_{\lambda}^0}|g''(0)|
\\
&=&\label{eq-thm-main2} \sqrt{\frac{\lambda}{2}+1+\frac 12\sup_{\omega\in\mathcal A_{\lambda}}\|\omega^*\|^2}+\frac 12\sup_{\omega\in\mathcal A_{\lambda}}\|\omega^*\|
\,.
\end{eqnarray}
Furthermore,
\begin{eqnarray}\label{eq-bound-f0}
\label{eq-thm-main3}\frac 12 \sup_{f\in\mathcal F^0_\lambda}|h''(0)|&=& \sqrt{\frac\lambda2+1+\frac 12\sup_{f\in\mathcal F_{\lambda}^0}|g''(0)|^2}\\
&=&\label{eq-thm-main4}\sqrt{\frac\lambda2+1+\frac 12\sup_{\omega\in\mathcal A_{\lambda}^0}|\omega'(0)|^2}\,.
\end{eqnarray}
\end{thm}
\par\smallskip
It was proved in \cite{HM2} that the Schwarzian norm of the harmonic Koebe function equals $19/2$. For $\lambda=19/2$ our result gives
\[
\frac 12 \sup_{f\in\mathcal F^0_\lambda}|h''(0)|=\frac 52 \, ,
\]
which would show that the order of $\mathcal{S}_H$ is equal to $3$ provided the harmonic Koebe function was extremal in the class for the
Schwarzian norm.

\par\smallskip


\section{Schwarzian derivative}

Let $f=h+\overline g$ be a locally univalent harmonic mapping in a simply connected domain $\Omega$ with dilatation $\omega=g^\prime/h^\prime$. In \cite{HM2}, the Schwarzian derivative $S_f$ of such a function $f$ was defined. If $f$ is sense-preserving, $S_f$ is given by
\begin{equation}\label{eq-Schwarzian}
S_f=Sh+\frac{\overline \omega}{1-|\omega|^2}\left(\frac{h''}{h'}\,\omega'-\omega''\right)
-\frac 32\left(\frac{\omega'\,\overline
\omega}{1-|\omega|^2}\right)^2\,.
\end{equation}

Several properties of this operator are the following:

\begin{enumerate}

\item[(i)] $S_f\equiv 0$ if and only if $f=\alpha T+\beta\overline{T}$, where $|\alpha|\neq|\beta|$ and $T$ is a M\"{o}bius transformation of the form
$$T(z)=\frac{az+b}{cz+d}\,,\quad ad-bc\neq 0\,.$$

\item[(ii)] Whenever $f$ is a sense-preserving harmonic mapping and $\phi$ is an analytic function such that the composition $f\circ\phi$ is well-defined, the Schwarzian derivative of $f\circ\phi$ can be computed using the \emph{chain rule}
    $$S_{f\circ\phi}=S_f(\phi)\cdot(\phi')^2+S\phi\,.$$

\item[(iii)] For any affine mapping $L(z)=az+b\overline z$ with $|a|\neq |b|$, we have that $S_{L\circ f}=S_f$. Note that $L$ is sense-preserving if and only if $|b|<|a|$.

\end{enumerate}

\par\smallskip
Consider now a sense-preserving harmonic mapping $f$ in the unit disk. Using the chain rule, we can see that for each $z\in\D$ $$|S_f(z)|= |S_{(f\circ \sigma_z)}(0)|\cdot(1-|z|^2)^2\,,$$ where $\sigma_z$ is any automorphism of the unit disk with $\sigma_z(0)=z$. The \emph{Schwarzian norm} $\|S_f\|$ of $f$ is defined by
\[
\|S_f\|=\sup_{z\in\D} |S_f(z)|\cdot(1-|z|^2)^2\,.
\]
\par
It is easy to check (using the chain rule again and the Schwarz-Pick lemma) that $\|S_{f\circ\, \sigma}\|=\|S_f\|$ for any automorphism of the unit disk $\sigma$. For further properties of $S_f$ and the motivation for this definition, see \cite{HM2}.


\section{Affine and Linear Invariant Families}

Let $\mathcal{S}_H$ denote the family of sense-preserving univalent harmonic mappings $f = h + \overline g$ on $\D$  normalized by $h(0) = 0$, $h'(0)=1$, and $g(0) = 0$. This family is affine and linear invariant. As usual, we use $\mathcal{S}_H^0$ to denote the subclass of functions in $\mathcal{S}_H$ that satisfy the further normalization $g'(0) = 0$. The family $\mathcal{S}_H$ is normal and $\mathcal{S}_H^0$ is compact (see \cite{CSS} or \cite{Dur-Harm}). Analogous results are obtained when dealing with the families $C_H$ and $C_H^0$ of convex harmonic mappings in $\mathcal{S}_H$ and $\mathcal{S}_H^0$, respectively.
\par\smallskip
Other examples of $\mathcal{AL}$ families of sense-preserving harmonic mappings are the stable harmonic univalent ($\mathcal{SHU}$) and the stable harmonic convex ($\mathcal{SHC}$) classes. A function $f=h+\overline g\in \mathcal{S}_H$ is {\it $\mathcal{SHU}$} (resp. {\it $\mathcal{SHC}$\,}) if $h+\lambda \overline g$ is univalent (convex) for every $|\lambda|=1$. These classes are linear invariant, and also affine because univalence or convexity are preserved under the
affine changes $A_\varepsilon$ as in \eqref{eq-affinetransform}. An important observation is that if the harmonic mapping $f$ has dilatation $\omega$, then  $F=A_{\omega(0)}(f)$ will have a dilatation vanishing at the origin.
\par
It is easy to check that $\alpha$($\mathcal{SHU}$)$=2$ and $\alpha$($\mathcal{SHC}$)$=1$ (see \cite{HM1}). In the next theorem, we obtain sharp bounds for the Schwarzian norm of functions in these classes.
\begin{thm}\label{cotas}
Let $f=h+\overline g$ be a locally univalent harmonic mapping defined in $\D$.
\begin{enumerate}
\item[(i)] If $f$ is a {\it $\mathcal{SHU}$} mapping, then $\|S_f\|\leq 6$.

\item[(ii)] If $f$ is a {\it {\it $\mathcal{SHC}$}} mapping, then $\|S_f\|\leq 2$.
\end{enumerate}
Both constants are sharp.
\end{thm}

\begin{proof}
It was shown in \cite{HM1} that if $f=h+\overline g\in$ $\mathcal{SHU}$, then $h+ag$ is univalent for all $|a|<1$; in particular, $h$ itself is univalent.
Assume that there exists $f\in\mathcal{SHU}$ with $\|S_f\|> 6$ and let $\omega$ be its dilatation. Then, there is a point $\zeta\in\D$ such that $|S_f(\zeta)|\cdot(1-|\zeta|^2)^2>6$.
Using the chain rule for the Schwarzian derivative and the affine invariance, we see that \begin{equation}\label{eq-prop}
|S_{A_{\omega_\zeta(0)}(K_\zeta(f))}(0)|=|S_f(\zeta)|\cdot(1-|\zeta|^2)^2>6\,,
\end{equation}
where $K_\zeta$ is the transformation defined by \eqref{eq-koebetransform}, $\omega_\zeta$ is the dilatation of the function $K_\zeta(f)$, and $A_{\omega_\zeta(0)}$ is as in  \eqref{eq-affinetransform}.
Since the dilatation of $K_\zeta(f)$ at the origin is $\omega_\zeta(0)$, we have that the dilatation of $A_{\omega_\zeta(0)}(K_\zeta(f))$ fixes the origin. Let $H$ denote the (univalent) analytic part of $A_{\omega_\zeta(0)}(K_\zeta(f))$; keeping in mind the definition \eqref{eq-Schwarzian} for the Schwarzian derivative, we see from \eqref{eq-prop} that
$$|S_{A_{\omega_\zeta(0)}(K_\zeta(f))}(0)|=|SH(0)|>6\, ,$$
which contradicts the univalence of $H$. This proves statement (i). The proof of (ii) follows the same argument, except for the fact that convex analytic mappings have Schwarzian norm bounded by 2 \cite{N2}.
\par
To prove that both constants are sharp, it is enough to consider the analytic functions
$$k(z)=\frac{z}{(1-z)^2}\quad\text{and}\quad s(z)=\frac 12\log\left(\frac{1+z}{1-z}\right)$$
that belong to the families of $\mathcal{SHU}$ and $\mathcal{SHC}$ mappings and have Schwarzian norms $\|S_k\|=6$ and $\|S_s\|=2$,  respectively.
\end{proof}
\par\smallskip

\par


\section{Admissible Dilatations }\label{sec-dilat}

In this section, we review some of the properties of hyperbolic derivatives of self-maps of the unit disk and determine the relation between hyperbolic norms of admissible dilatations in $\mathcal F_\lambda$ and the parameter $\lambda$ itself.

\subsection{The hyperbolic derivative}
Let $\omega$ be a self-map of the unit disk, this is, an analytic function in $\D$ with $\omega(\D)\subset\D$. The \emph{hyperbolic derivative} of such function $\omega$ is
\[
\omega^*(z)=\frac{\omega^\prime(z)\cdot (1-|z|^2)}{1-|\omega(z)|^2}\,,\quad z\in\D\,.
\]
\par
From Schwarz's lemma we see that $|\omega^*|\leq 1$ in $\D$, and that if there exists $z_0\in\D$ with $|\omega^*(z_0)|=1$, then $\omega$ is an automorphism of the unit disk and $|\omega^*|\equiv 1$ in $\D$. Of course, there are self-maps of $\D$ with hyperbolic norm equal to 1 which are not automorphisms. There are
examples such as $\omega(z)=(z+1)/2$, but also, every finite Blaschke product has hyperbolic norm equal to $1$ \cite{Heins}. See also \cite{M} for other examples.
\par\smallskip
Given two self-maps $\omega$ and $\varphi$ of the unit disk, the \emph{chain rule for the hyperbolic derivative} holds:
\[
\left(\varphi\circ \omega\right)^*(z)= \varphi^*(\omega(z))\cdot \omega^*(z)\,.
\]
In particular, if $\sigma$ is an automorphism of $\D$, then $|(\sigma\circ\omega)^*|\equiv |\omega^*|$ in the unit disk, hence $\|(\sigma\circ\omega)^*\|=\|\omega\|$.
\par\smallskip
Just like in \cite{HM3}, the so-called \emph{lens-maps} $\ell_\alpha$ will be of particular interest. For $0<\alpha<1$, the mapping $\ell_\alpha$ is defined by
\begin{equation}\label{eq-lensmap}
\ell_\alpha(z)=\frac{\ell(z)^\alpha-1}{\ell(z)^\alpha+1}\,,
\end{equation}
where $\ell(z)=(1+z)/(1-z)$. The hyperbolic norm of $\ell_\alpha$ was computed in \cite{HO} to be $\|\ell_\alpha^*\|=\a$. Moreover, $|\ell_\alpha^*(r)|=\alpha$ for all real numbers $0\leq r<1$.

\par\smallskip
\subsection{Admissible dilatations and norms}

Recall that we say that a self-map $\omega$ of the unit disk belongs to the family of admissible dilatations $\mathcal A_\lambda^0$ if there exists $f\in \mathcal F_\lambda^0$ with dilatation $\omega$. For any such $f$ and any $\a\in\D$, the affine transformation $f_\a=A_{-\a}(f)=f+\overline {\a f} \in\mathcal F_\lambda$. Its dilatation $\omega_\a$ is given by
\[
\omega_\a=\sigma_\a\circ\omega\,,
\]
where $\sigma_\a$ is the automorphism in the unit disk defined by
\begin{equation}\label{eq-autom}
\sigma_\a(z)=\frac{\a+z}{1+\overline \a z}\,,\quad z\in\D\,.
\end{equation}
In other words, whenever $\omega\in\mathcal A_\lambda^0$ and $\a\in\D$, then $\sigma_\a\circ\omega \in \mathcal A_\lambda$.
\par
On the other hand, if $f\in \mathcal F_\lambda$ that has dilatation $\omega$ satisfying $\omega(0)=\a$, then
\[
F=\frac{f-\overline{\a f}}{1-|\a|^2}\in \mathcal F_\lambda^0\,,
\]
with a resulting new dilatation $\omega_F=\sigma_\a^{-1}\circ \omega=\sigma_{-\a}\circ\omega$. These relations establish a correspondence between the families $\mathcal A_\lambda^0$ and $\mathcal A_\lambda$.

\par\smallskip
It is easy to verify that for any value of $\lambda$, there exists $\omega\in\mathcal A_\lambda$ with $\|\omega^*\|=0$ (just consider the identity function $I(z)=z$ in the unit disk which belongs to $\mathcal F_\lambda$ for all $\lambda\geq 0$). The following theorem characterizes the values of $\lambda$ for which
a dilatation in $\mathcal{A}_{\lambda}$ can have hyperbolic norm 1.

\begin{thm}\label{prop-dilat}
 The following conditions are equivalent.
\begin{itemize}
 \item[(i)] $\lambda\geq 3/2$.
 \item[(ii)] There exists $\omega \in\mathcal A_\lambda^0$ with $|\omega^\prime(0)|= 1$.
 \item[(iii)] The set $\{\lambda\cdot I\colon |\lambda|=1\}$ is contained in $\mathcal A_\lambda^0$. In particular, the identity function $I$ is an admissible dilatation in $\mathcal F_\lambda^0$.
 \item[(iv)] Every automorphism $\sigma$ of the unit disk is an admissible dilatation in $\mathcal F_\lambda$.
 \item[(v)] There exist $\omega \in\mathcal A_\lambda$ with $\|\omega^*\|= 1$.
\end{itemize}
\end{thm}
\begin{pf} The scheme of the proof is to show that (i) $\Rightarrow$ (ii) $\Rightarrow$ (iii) $\Rightarrow$ (i). Then we prove that (iii) $\Longleftrightarrow$ (iv),
and finally we see that (iv) $\Rightarrow$ (v) $\Rightarrow$ (ii).

\par\smallskip
To show that (i) $\Rightarrow$ (ii), we consider the function $f=z+\frac 12 \overline z^2$. Note that the dilatation $\omega$ of $f$ equals the identity function $I$ so that $\omega'(0)=1$. Since
\[
S_f(z)=- \frac{3\,\overline z^2}{2\,(1-|z|^2)^2}\,,
\]
we have $\|S_f\|= 3/2$. Thus, $f\in \mathcal{F}^0_\lambda$ for all $\lambda\geq 3/2$ and (ii) holds.
\par\smallskip
Let us now check that (ii) $\Rightarrow$ (iii). To do so, just note that given any harmonic function  $f=h+\overline g\in\mathcal{F}^0_\lambda$ with dilatation $\omega$ and any $\mu\in\partial\D$, the functions $f_{\mu}=h+\overline{\mu g}$ belong to $\mathcal F_\lambda^0$ as well since $S_f =S_{f_\mu}$ for all such $\mu$ and the dilatation $\omega_\mu$ of $f_\mu$ is equal to $\omega_\mu(z)=\mu\,\omega(z)$ (which satisfies $\omega_\mu(0)=0$).
\par
Now, if we assume that there exists $\omega \in\mathcal A_\lambda^0$ with $|\omega^\prime(0)|= 1$ we have, by the Schwarz lemma, that $\omega$ is a rotation of the disk (i.e. $\omega(z)=\lambda z$ for some $|\lambda|=1$). Therefore, using those functions $f_\mu$ we immediately get that (iii) holds.
\par\smallskip
Suppose that (iii) is satisfied. This is, there is a function  $f=h+\overline g \in \mathcal F_\lambda^0$ with dilatation $\omega=I$. We are going to prove that $\lambda \geq 3/2$. Indeed, we will obtain a stronger result in view of inequality \eqref{eq-ineq2} below. More specifically, we will see that (iii) implies that the order $\alpha(\mathcal F_\lambda)$ of the family $\mathcal F_\lambda$ is $2$ at least, so that by \eqref{eq-ineq2} we obtain (i) (note that the supremum that appears there equals $1$ whenever $I\in\mathcal A_\lambda^0$) .
\par
Using the affine invariance property of $\mathcal F_\lambda$, we have that whenever $f\in\mathcal F_\lambda^0$, the function $F_\varepsilon=A_\varepsilon(f)=f-\overline{\varepsilon f}\in \mathcal F_\lambda$ for all $\varepsilon \in \D$, where $A_\varepsilon$ is the transformation defined in \eqref{eq-affinetransform}. The analytic part in the canonical decomposition of $F_\varepsilon$ equals $h_\varepsilon=h-\overline\varepsilon g$ so that $h_\varepsilon^\prime(z)=h^\prime(z)\cdot(1-\overline\varepsilon I)$. Therefore, just applying the same arguments that appear in the proof of the theorem on \cite[p. 97]{Dur-Harm}, we get that for all $z\in\D$
\begin{equation}\label{eq-prop1}
|h^\prime(z)|\cdot|1-\varepsilon z|\geq \frac{(1-|z|)^{\alpha(\mathcal F_\lambda)-1}}{(1+|z|)^{\alpha(\mathcal F_\lambda) + 1}}
\end{equation}
whenever $\varepsilon<1$ and hence, for all $|\varepsilon|\leq 1$. Thus given any $z\neq 0$ in the unit disk, we can choose $\varepsilon=\overline z/|z|$ to get from \eqref{eq-prop1} that
\[
|h^\prime(z)|\geq \frac{(1-|z|)^{\alpha(\mathcal F_\lambda)-2}}{(1+|z|)^{\alpha(\mathcal F_\lambda) + 1}}\,,
\]
an inequality that obviously works for $z=0$. As a consequence, since $h$ is locally univalent in $\D$, we see that the analytic function $1/h^\prime$ in the unit disk satisfies
\[
\frac{1}{|h^\prime(z)|}\leq \frac{(1+|z|)^{\alpha(\mathcal F_\lambda)+1}}{(1-|z|)^{\alpha(\mathcal F_\lambda) -2}}\,,
\]
which implies, by the maximum modulus principle, that $\alpha(\mathcal F_\lambda)\geq 2$ (otherwise we would get $1/|h^\prime(0)|<1$, which is absurd since $h'(0)=1$). This shows that (iii) $\Rightarrow$ (i).
\par\smallskip
We continue with the proof of the theorem by showing that (iii) $\Leftrightarrow$ (iv). Note that since every rotation is an automorphism of the unit disk that fixes the origin, only the implication (iii) $\Rightarrow$ (iv) is needed to check the equivalence between these two statements. Let us assume then that the set $\{\lambda I\colon |\lambda|=1\}\subset \mathcal A_\lambda^0$. Then, we can use that $\mathcal F_\lambda$ is affine invariant and the functions $f_\mu$ defined above to see that for any $\a$ in the unit disk and any $\lambda, \mu\in \partial\D$, the functions
\[
\mu\sigma_\a\circ(\lambda I)=\mu\cdot \frac{\a+\lambda z}{1+\overline \a \lambda z}= \mu\lambda \cdot \frac{\overline \lambda \a+z}{1+\overline {\overline\lambda \a} z}\in \mathcal A_\lambda\,.
\]
In other words, for any $|\eta|=1$ and any $\beta\in\D$, the mappings
\begin{equation}\label{eq-autom-generalform}
\eta\cdot \frac{\beta+z}{1+\overline\beta z} \in \mathcal A_\lambda.
\end{equation}
The Schwarz lemma says that any automorphism of the unit disk has the form \eqref{eq-autom-generalform} so that (iv) holds.
\par\smallskip
Finally, we are show the equivalence between statements (iv) and (v). Once more, only one of the implications in the equivalence is non-trivial (recall that any automorphism of the unit disk has hyperbolic norm equal to one). Concretely, we just need to see that (v)  $\Rightarrow$ (iv). Since at this point we have proved that (i), (ii), (iii), and (iv) are equivalent, it suffices to check that (v) $\Rightarrow$ (ii). To do so, let $\omega\in \mathcal A_\lambda$ have the property that $\|\omega^*\|=1$. Then, by the definition of the hyperbolic norm, there exists a sequence of points in the unit disk $\{z_n\}$, say, such that
\[
\lim_{n\to\infty} |\omega^*(z_n)|=1.
\]
The fact that $\omega\in \mathcal A_\lambda$ means that there is a function $f=h+\overline g\in\mathcal F_\lambda$ with dilatation $\omega$. Using the transformations $K_{z_n}$ and $A_{\omega_n(0)}$ defined by \eqref{eq-koebetransform} and \eqref{eq-affinetransform}, respectively, where $\omega_n$ is the dilatation of $K_{z_n}(f)$, we obtain a sequence of harmonic mappings
\[
f_n=A_{\omega_n(0)}(K_{z_n}(f))
\]
with dilatations
\[
\gamma_n=\sigma_{-\omega_n(0)}\circ\omega_n=  \sigma_{-\omega_n(0)}\circ[\lambda_n\cdot(\omega\circ\sigma_{z_n})]\,,
\]
where $\lambda_n=h'(z_n)/\overline{h'(z_n)}$ and $\sigma_\a$ is again the automorphism of $\D$ defined by \eqref{eq-autom}.
\par
Note that $\gamma_n(0)=0$, so that each element in the sequence $f_n\in \mathcal F_\lambda^0$. Moreover, a straightforward computation shows that we also have
$|\gamma_n^\prime(0)|=|\omega^*(z_n)|$.
\par
Now, we argue as on \cite[pp. 81-82]{Dur-Harm} but in this case, instead of the argument principle, we use the fact that whenever $f_n\to f_0$ uniformly on compact subsets of $\D$ we have that
for each $z$ in the unit disk
\[
S_{f_n}(z)\to S_{f_0}(z)\,.
\]
This shows that $\mathcal{F}_\lambda^0$ is a normal and compact family. Thus, there exists a subsequence (that we rename $\{f_n\}$ again) of the sequence $\{f_n\}$ that converges to $f_0\in\mathcal F_\lambda^0$ uniformly on compact subsets in the unit disk. The dilatation $\gamma_0$ of the limit function $f_0$ is
\[
\gamma_0(z)=\lim_{n\to\infty} \gamma_n(z)\,,
\]
so that it satisfies
\[
|\gamma^\prime_0(0)|=\left|\lim_{n\to\infty} \gamma_n'(0)\right|=\lim_{n\to\infty} \left|\omega^*(z_n)\right|=1\,.
\]
This proves that (v) $\Rightarrow$ (ii) and finishes the proof of the theorem.
\end{pf}
\par\smallskip

From Theorems 1 and 3 we derive the following important corollary.
\begin{cor}\label{cor-lambda-grande}
If $\lambda\geq 3/2$ then
$$
\alpha\left( \mathcal F_\lambda\right) = \sqrt{\frac{\lambda}{2}+\frac 32}+\frac 12
$$
and
$$
\frac 12 \sup_{f\in\mathcal F^0_\lambda}|h''(0)|= \sqrt{\frac\lambda2+\frac 32}\,.
$$
\end{cor}

Even though for $\lambda <3/2$ we have been unable to determine the value of $\sup_{\omega\in\mathcal A_\lambda} \|\omega^*\|$, one can
rewrite Theorem 1 in the form
\begin{equation}\label{dilatacion chica}
\sup_{\omega\in\mathcal A_\lambda} \|\omega^*\|=-2\alpha(\mathcal F_\lambda)+ \sqrt{8\alpha(\mathcal F_\lambda)^2-2\lambda-4}\,.
\end{equation}

\section{Proof of the Main Theorem}
We will divide the proof of Theorem 1 in three different parts.

\subsection{A lemma}
We begin by proving that \eqref{eq-thm-main2} and \eqref{eq-thm-main4} hold.
\begin{lem}\label{lem-equal}
For any positive real number $\lambda$,
\begin{itemize}
\item[(i)]
\[
\sup_{f=h+\overline g \in\mathcal F_\lambda^0} |g''(0)|= \sup_{\omega \in\mathcal A_\lambda^0} |\omega'(0)|\,.
\]
\item[(ii)]
\[
\sup_{\omega \in\mathcal A_\lambda^0} |\omega'(0)|= \sup_{\omega \in\mathcal A_\lambda^0} \|\omega^*\|=\sup_{\omega \in\mathcal A_\lambda} \|\omega^*\|\,.
\]
\end{itemize}
\end{lem}
\begin{pf}
The proof of (i) is trivial: just recall that the dilatation $\omega$ of $f=h+\overline g \in\mathcal F_\lambda^0$  equals $\omega=g'/h'$ and satisfies $\omega(0)=0$. Thus, $g'=\omega h'$ and hence $g''(0)=\omega'(0) h'(0)+\omega(0)h''(0)=\omega'(0)$, since $h'(0)=1$.
\par
To prove (ii), take an arbitrary $f=h+\overline g \in\mathcal F_\lambda^0$ with dilatation $\omega$. The transformations \eqref{eq-koebetransform} and \eqref{eq-affinetransform} produce new functions that also belong to this family. Concretely, given any such function $f$, we define the mappings
\[
f_\z=A_{\omega_\zeta(0)}(K_{\z}(f))=h_\z+\overline{g_\z}\,,\quad \z\in\D\,,
\]
where $\omega_\zeta$ is the dilatation of $K_\zeta(f)$. These mappings $f_\z$ belong to $\mathcal F_\lambda^0$ and have dilatations
\[
\gamma_\z= \sigma_{-\omega_\zeta(0)}\circ\omega_\zeta\in \mathcal A_\lambda^0\,,
\]
where $\omega_\zeta=\lambda_\zeta\cdot(\omega\circ\sigma_\z)$ for certain $|\lambda_\zeta|=1$ and appropriate automorphisms $\sigma_\a$ of the form \eqref{eq-autom}. Since
\[
|\gamma_\z^\prime(0)|=|\omega^*(\z)|
\]
and $\z$ is any arbitrary point in $\D$, we conclude that the first equality in (ii) holds. The second inequality is an easy consequence of the fact that the correspondence between the families $\mathcal A_\lambda$ and $\mathcal A_\lambda^0$ is realized using pre-composition with automorphisms of the unit disk (an operation that preserves the hyperbolic norm).
\end{pf}

\subsection{Upper bounds}
The aim of this section is to prove the following inequalities:
\begin{equation}\label{eq-ineq1}
\frac 12 \sup_{f\in\mathcal F^0_\lambda}|h''(0)|
\leq \sqrt{\frac\lambda2+1+\frac 12\sup_{\omega\in\mathcal A_{\lambda}^0}|\omega'(0)|^2}
\end{equation}
and
\begin{equation}\label{eq-ineq2}
\alpha\left( \mathcal F_\lambda\right) \leq \sqrt{\frac{\lambda}{2}+1+\frac 12\sup_{\omega\in\mathcal A_{\lambda}}\|\omega^*\|^2}+\frac 12\sup_{\omega\in\mathcal A_{\lambda}}\|\omega^*\|\,.
\end{equation}

To prove \eqref{eq-ineq1}, let us agree with the notation
\[
S=\frac 12 \sup_{f\in\mathcal{F}^0_\lambda}|h''(0)|\,.
\]
Given $f\in\mathcal{F}^0_\lambda$, the function $F$ defined by $F(z)=\overline \lambda f(\lambda z) \in \mathcal{F}^0_\lambda$ as well. Therefore, we see that
\[
S=\frac 12 \sup_{f\in\mathcal{F}^0_\lambda}{\rm Re}\, \{h''(0)\}\,.
\]
As it was mentioned in the proof of Theorem~\ref{prop-dilat}, the family $\mathcal{F}_\lambda^0$ is normal and compact. Therefore, there exists a function $f_0=h_0+\overline{g_0}$ in $\mathcal{F}_\lambda^0$ with dilatation $\omega_0$, where $$h_0(z)=z+a_2z^2+a_3z^3+\ldots,\quad g_0(z)=b_2z^2+b_3z^3+\ldots\,,$$ such that
$${\rm Re}\,\{a_2\}=|a_2|=S\,.$$
\par
Take an arbitrary point $\zeta\in\D$ and consider, once more, the transformations \eqref{eq-koebetransform} and \eqref{eq-affinetransform} to produce the family
$$F_\zeta=A_{\omega_{\zeta}(0)}(K_\z(f))=h_\zeta^*+\overline{g_\zeta^*}$$
of functions that are in $\mathcal{F}_\lambda^0$.
\par\smallskip
As it is shown on \cite[p.102]{Dur-Harm}, the Taylor coefficients $a_n^\ast$ of $h_\zeta^*$ satisfy
\begin{equation*}
\label{eq-label a_n*} a_n^\ast=a_n+[(n+1)a_{n+1}-2a_2a_n]\zeta-[2\overline{b_2}b_n+(n-1)a_{n-1}]\overline{\zeta}+o(|\zeta|)\,.
\end{equation*}
Then, we get that
\begin{equation}\label{marty} 3a_3-2a_2^2-2|b_2|^2-1=0\,.
\end{equation}
\par
Now, using the equations $S_{f_0}(0)=Sh_0(0)=6(a_3-a_2^2)$, we have by (\ref{marty}) that ${\rm Re}\,\{S_{f_0}(0)\}=4|b_2^2|+2-2{\rm Re}\,\{a_2^2\}$. Hence, bearing in mind $S={\rm Re}\,\{a_2\}=|a_2|$, we obtain
\begin{eqnarray*}
S^2&=&|a_2|^2\leq \frac 12|S_{f_0}(0)|+1+2|b_2|^2\\
&\leq& \frac{\lambda}{2}+1+\frac 12 \sup_{\omega\in\mathcal{A}^0_\lambda}|\omega'(0)|^2.
\end{eqnarray*}
This proves \eqref{eq-ineq1}. To show that \eqref{eq-ineq2} holds we can argue as follows. Take any function $f=h+\overline g \in\mathcal F_\lambda$ with dilatation $\omega$ and consider the affine transformation $A_{\omega(0)}$ as in \eqref{eq-affinetransform} to get the function
$$F=A_{\omega(0)}(f)= H+\overline G\in \mathcal{F}_\lambda^0\,.$$
Note that $f=F+\overline{\omega(0) F}$. Thus, the analytic part of $f$ equals $H+\overline{\omega(0)}G$. Using also that $\omega(\D)\subset\D$, and Lemma~\ref{lem-equal}, we get
\begin{eqnarray*}
|h''(0)|&=&|H''(0)+\overline{\omega(0)}G''(0)|\leq 2S+|\omega(0)|\, \sup_{f\in\mathcal F_\lambda^0} |g''(0)|\\
&\leq& 2S+\, \sup_{\omega\in\mathcal A_\lambda} \|\omega^*\| \,,
\end{eqnarray*}
which gives \eqref{eq-ineq2}.

\subsection{Equalities}
In order finish the proof of Theorem 1, we will show that \eqref{eq-ineq1} and \eqref{eq-ineq2} are actual equalities.
\par
To see that equality holds in \eqref{eq-ineq1}, we will exhibit a function $f_0=h_0+\overline{g_0} \in \mathcal F_\lambda^0$ with dilatation $\omega_0$ satisfying the following properties:
$$
\omega_0^\prime(0)=\sup_{\omega\in\mathcal A_\lambda^0} \|\omega^*\|\quad\text{and}\quad h_0''(0)=2\cdot \sqrt{\frac\lambda 2+1+\omega_0^\prime(0)}\,.
$$
This will show  that \eqref{eq-thm-main3} holds. Furthermore, for any $0<r<1$, the mappings $f_r=f_0+r\overline{f_0}\in\mathcal F_\lambda$ and have analytic parts
$h_r=h_0+rg_0$, for which
\begin{eqnarray*}
h_r''(0)&=&h_0''(0)+r g_0''(0)= \sup_{f\in\mathcal F_\lambda^0} |h''(0)|+ r \omega_0'(0)\\ &=& \sup_{f\in\mathcal F_\lambda^0} |h''(0)|+ r \sup_{\omega\in\mathcal A_\lambda^0} \|\omega^*\| \\ &\to& \sup_{f\in\mathcal F_\lambda^0} |h''(0)|+ \sup_{\omega\in\mathcal A_\lambda^0} \|\omega^*\|\,,
\end{eqnarray*}
as $r\to 1^-$. This shows that equality also holds in \eqref{eq-ineq2}, proving thus \eqref{eq-thm-main}.
\par
The function $f_0$ will be constructed in the next, final section.


\section{The mapping $f_0$}

Given $\lambda\geq 0$, let us denote by $R_\l=\sup_{\omega\in\mathcal A_\lambda} \|\omega^*\|$. If $\lambda=0$, then by \cite[Cor. 2]{HM2} every function in $\mathcal F_\lambda$ has the form $h+a\overline h$, where $a\in\D$ and
\[
h(z)=\frac{z}{1+bz}\,,\quad |b|<1\,.
\]
Hence $R_\l=0$. We now analyze the cases $\lambda >0$.

\subsection{Harmonic mappings in $\mathcal F_\lambda^0$ when $\lambda>0$}
By Theorem~\ref{prop-dilat}, we have that $\lambda\geq 3/2$ if and only if $R_\lambda =1$. Equation (\ref{dilatacion chica}) gives the exact value of $R_\lambda$ in terms of $\lambda$ and $\alpha(\mathcal F_\lambda)$. The following example allows us to estimate the value of $R_\lambda$ in terms of $\lambda$ for $0<\lambda<3/2$.
\par\medskip

Given $0<\lambda < 3/2$, we can write $\lambda=3s^2/2$ for some $0<s<1$. Consider the functions $f_r=z+\frac12r\overline{z}^2$, where $r\in (0,1)$. The Schwarzian norm of  $f_r$ equals
\begin{equation}\label{eq-ex1}
\|S_{f_r}\|=\frac 32 \sup_{z\in\D}\frac{r^2|z|^2}{(1-|rz|^2)^2}\cdot(1-|z|^2)^2\leq \frac 32\, r^2\,.
\end{equation}
Hence, $f_r\in \mathcal F_\lambda^0$ whenever $r\leq s$. This implies that $R_\l\geq r=||\omega_r^*||$, where $\omega_r$ is the dilatation of $f_r$. In particular, this
implies that if
$\lambda \in (0, 3/2)$,
\begin{equation}\label{eq-upperR}
R_\l>\sqrt{\frac{2\lambda}{3}}\,.
\end{equation}
Note that we have strict inequality in \eqref{eq-upperR} since the supremum that appears in \eqref{eq-ex1} is strictly less than $r^2$. Also, that $R_\l>0$ if $\l>0$.

\par\medskip

\subsection{The extremal}
Let us introduce the  notation $\alpha=\alpha(\mathcal F_\lambda)$ and $R=R_\lambda=\sup\{\|\omega^*\|\colon \omega\in \mathcal A_\lambda\}$.
For $\lambda>0$ we consider the mapping $f_0=h_0+\overline{g_0}$, where $h_0$ and $g_0$ solve the linear system of equations
\begin{equation}\label{eq-function}
\left\{
\begin{array}{c}
h_0-g_0=\varphi_a\\
\omega_0=g_0^\prime/h_0^\prime=\ell_R
\end{array}\right.\,,\quad\quad h_0(0)=g_0(0)=0\,,
\end{equation}
where
\[
a=\sqrt{\frac\lambda 2+1+\frac12 R^2}-\frac R2\, ,
\]
and the function $\varphi_a$ is the generalized Koebe function defined by \eqref{eq-sharp-functions}. Also, $\ell_1$ is the identity mapping $I$ and for $0<R<1$, $\ell_R$ is the lens-map \eqref{eq-lensmap}. Note that $h_0'(0)=1-g_0'(0)=0$ and that $f_0$ is a locally univalent mapping in the unit disk since the dilatation is a self-map of $\D$ and $h_0$ is locally univalent.

\par\smallskip
By \eqref{eq-function}, we have $h_0^\prime(1-\ell_R)=\varphi_a^\prime$ which implies
\[
\frac{h_0''}{h_0^\prime}=\frac{\varphi_a''}{\varphi_a^\prime}+\frac{\ell_R'}{1-\ell_R}\, ,
\]
so that
\[
Sh_0 = S\varphi_a +\frac{\ell_R''}{1-\ell_R}+\frac 12\frac{(\ell_R')^2}{(1-\ell_R)^2} - \frac{\ell_R'}{1-\ell_R}\cdot\frac{\varphi_a''}{\varphi_a^\prime}\,,
\]
which gives
\[
Sh_0 (0)=2(1-a^2) + \frac{R^2}{2} - 2a R = 2+\frac{R^2}{2} -2a R-2 a^2\,.
\]
Define the function
\[
\psi(x)=2+\frac{R^2}{2} -2x R-2 x^2\,.
\]
Note that
\[
\psi\left(\sqrt{1+\frac{R^2}{2}}-\frac R2\right)=0\,.
\]
Also, that $\psi^\prime <0$ which implies, since
\[
a> \sqrt{1+\frac{R^2}{2}}-\frac R2\,,
\]
that $\psi(a)<0$. This means that $|Sh_0 (0)|=2 a^2 + 2a R -2-R^2/2$, which is easily seen to be equal to $\lambda$. Since $\omega_0(0)=0$, we also have that $|Sf_0 (0)|=\lambda$. To show that $f_0\in\mathcal F_\lambda^0$, we just need to check $\lambda=|S_{f_0}(0)|=\|S_{f_0}\|$. To do so, we compute the Schwarzian derivative of $f_0$ which, according to \eqref{eq-Schwarzian}, equals
\begin{eqnarray}
\nonumber S_{f_0}&=& Sh_0+\frac{\overline{\ell_R}}{1-|\ell_R|^2}\left( \frac{h_0''}{h_0^\prime}\ell_R'-\ell_R''\right)-\frac 32\left(\frac{\overline{\ell_R} \ell_R'}{1-|\ell_R|^2}\right)^2\\
&=& \label{eq-form1} S\varphi_a +\frac{\ell_R''}{1-\ell_R}+\frac 12\frac{(\ell_R')^2}{(1-\ell_R)^2} - \frac{\ell_R'}{1-\ell_R}\cdot\frac{\varphi_a''}{\varphi_a^\prime}
\\
&+& \nonumber \frac{\overline{\ell_R}}{1-|\ell_R|^2}\left( \frac{\varphi_a''}{\varphi_a^\prime}\ell_R'+\frac{(\ell_R')^2}{1-\ell_R}-\ell_R''\right)-\frac 32\left(\frac{\overline{\ell_R} \ell_R'}{1-|\ell_R|^2}\right)^2\,.
\end{eqnarray}
Since for any complex number $z$,
\[
\frac{1}{1-z}-\frac{\overline z}{1-|z|^2}=\frac{1-\overline z}{(1-z)(1-|z|^2)}\,,
\]
we get from \eqref{eq-form1}
\begin{eqnarray}
S_{f_0}&=& \nonumber S\varphi_a+\frac{1-\overline{\ell_R}}{1-\ell_R}\cdot \frac{\ell_R''}{1-|\ell_R|^2}-\frac{1-\overline{\ell_R}}{1-\ell_R}\cdot\frac{\ell_R'}{1-|\ell_R|^2}\cdot  \frac{\varphi_a''}{\varphi_a^\prime}\\
&+&\nonumber \frac {(\ell_R')^2}{2}\cdot\left[\left(\frac{1+\overline{\ell_R}-2|\ell_R|^2}{(1-\ell_R)(1-|\ell_R|^2)}\right)^2-4\cdot \left(\frac{\overline{\ell_R}}{1-|\ell_R|^2}\right)^2\right]\\
&=&\label{eq-form2} S\varphi_a+F_R\cdot \left(\frac{\ell''_R}{\ell'_R}-\frac{\varphi_a''}{\varphi_a^\prime}\right)+\frac{F_R^2}{2}\cdot\frac{1+3\overline{\ell_R}-4|\ell_R|^2}{1-\overline{\ell_R}}\,,
\end{eqnarray}
where
\[
F_R=\frac{1-\overline{\ell_R}}{1-\ell_R}\cdot\frac{\ell_R'}{1-|\ell_R|^2}\,.
\]
Let us write
\begin{equation}\label{eq-wybeta}
w=\left(\frac{1+z}{1-z}\right)^R\quad \text{and}\quad \beta=\frac{w}{Re w}\,.
\end{equation}
Then, $F_R(z)=R\beta/(1-z^2)$,
\[
\frac{\ell''_R(z)}{\ell'_R(z)}-\frac{\varphi_a''(z)}{\varphi_a^\prime(z)}=\frac{-2a}{1-z^2}+\frac{2R}{1-z^2}\cdot\frac{1-w}{1+w}\,,
\]
and
\[
\frac{1+3\overline{\ell_R}-4|\ell_R|^2}{1-\overline{\ell_R}}=\frac{5{\rm Re}\{w\}-3(1+i {\rm Im}\{w\})}{1+w}=\frac{8{\rm Re}\{w\}}{1+w}-3\,.
\]
Thus, we get from \eqref{eq-form2}
\begin{eqnarray*}
\nonumber S_{f_0}(z)\cdot(1-z^2)^2&=& 2(1-a^2)-2\a R\b-\frac{3R^2\b^2}{2} \\
&+&\nonumber 2R^2\b\cdot \frac{1-w}{1+w}+4 R^2\b^2\cdot\frac{{\rm Re}\{w\}}{1+w}\\
&=& 2(1-a^2)+2R(R-a)\b-\frac{3R^2\b^2}{2}\,,
\end{eqnarray*}
a formula that we use to show
\begin{align}\label{eq-formschwarzian}
\nonumber|S_{f_0}(z)|& \cdot(1-|z|^2)^2\\ &=\left|2(1-a^2)+2R(R-a)\b-\frac{3R^2\b^2}{2}\right|\cdot\left(\frac{1-|z|^2}{|1-z^2|}\right)^2\,.
\end{align}
\par\smallskip
Consider a real number $\gamma$ with $0 \leq \gamma < \pi/2$ and define the curves
\[
C_\gamma=\left\{z\in\D\colon {\rm Arg}\left(\frac{1+z}{1-z}\right)=\gamma\right\}\,.
\]
Note that $C_0$ is equal to the real diameter $(-1,1)$ and for $\gamma\neq 0$, $C_\gamma$ is a circular arc passing trough the points $-1$ and $1$.
\begin{lem}\label{lem-pf1}
The quantity $|S_{f_0}(z)|\cdot(1-|z|^2)^2$ is constant on the curves $C_\gamma$, $0\leq \gamma<\pi/2$.
\end{lem}
\begin{pf}
Take any $\gamma\in [0,\pi/2)$ and let $z\in C_\gamma$. Then there exists a (positive) real number $t$ such that
\[
\frac{1+z}{1-z}=te^{i\gamma}\,.
\]
A straightforward computation gives
\[
\b=\frac{t^R e^{iR \gamma}}{t^R\,\cos(R\gamma)} = 1+i\tan(R\gamma)\quad \text{and}\quad \frac{1-|z|^2}{|1-z^2|}=\cos\gamma\,.
\]
The proof of the lemma follows from \eqref{eq-formschwarzian}.
\end{pf}
\par\smallskip
Lemma~\ref{lem-pf1} shows, in particular, that whenever $r\in (0,1)$,
$|S_{f_0}(r)|(1-r^2)^2=|S_{f_0}(0)|\equiv\lambda\,.$
\par\smallskip
Note that $S_{f_0}(\overline z)=\overline{S_{f_0}(z)}$ for all $z\in\D$. Moreover, it is easy to check that any radius $\{re^{i\theta},\ 0<r<1\}$, with $0<\theta<\pi$, intersects every $C_\gamma$ with $\gamma>0$ at exactly one point. Therefore, we conclude that
\begin{equation}\label{eq-imag}
\|S_{f_0}\|=\sup_{0\leq r<1} |S_{f_0}(ir)|\cdot(1-r^2)^2\,.
\end{equation}
\begin{lem}\label{lem-pf2}
The Schwarzian norm of $f_0$ equals
\[
\|S_{f_0}\|=|S_{f_0}(0)|=\lambda\,.
\]
\end{lem}
\begin{pf}
According to \eqref{eq-imag}, it suffices to show
\[
\sup_{0\leq r<1} |S_{f_0}(ir)|\cdot(1-r^2)^2=|S_{f_0}(0)|\,.
\]
We use \eqref{eq-formschwarzian} to write
\begin{align}
\nonumber\Phi(r)=|S_{f_0}(ir)|& \cdot(1-|r|^2)^2\\
\label{eq-p1} &=\left|2(1-a^2)+2R(R-a)\b_r-\frac{3R^2\b_r^2}{2}\right|\cdot\left(\frac{1-r^2}{1+r^2}\right)^2\\
&\nonumber =\label{eq-pf2}|\phi\circ\b_r|\cdot\left(\frac{1-r^2}{1+r^2}\right)^2\,,
\end{align}
where $\phi(x)=A+Bx+Cx^2$, with $A=2(1-a^2)$, $B=2R(R-a)$, and $C=-3R^2/2$; and, by \eqref{eq-wybeta},
\begin{equation}\label{eq-betaygamma}
\b_r=1+i\tan(R\gamma_r) \quad\text{with}\quad \cos \gamma_r=(1-r^2)/(1+r^2)\,.
\end{equation}
We are to check
\begin{equation}\label{eq-sup}
\sup\{ \Phi(r)\colon 0\leq r <1\}=\Phi(0)\,.
\end{equation}
Instead of proving \eqref{eq-sup}, we consider the equivalent problem of showing
\begin{equation}\label{eq-sup2}
\sup\{ \Phi^2(r)\colon 0\leq r <1\}=\Phi^2(0)\,.
\end{equation}
The advantage of this new reformulation is that, as the reader may check, we can write
\[
|\phi\circ\b_r|^2=\widetilde{A}+\widetilde{B}|\b_r|^2+\widetilde{C}|\b_r|^4\,,
\]
with $\widetilde{A}=A^2+2AB+4AC$, $\widetilde{B}=B^2+2BC-2AC$, and $\widetilde{C}=C^2$. Note that $\widetilde{A}+\widetilde{B}+\widetilde{C}=\lambda^2$. In fact,
\begin{equation}\label{eq-Atilde}
\widetilde{A}=4(1-a^2)\left(1-(a+R)^2\right)
\end{equation}
 (with $a+R>1$, by the definition of $a$),
\begin{equation}\label{eq-BCtilde}
\widetilde{B}=2R^2(3-aR-a^2-R^2)\,,\quad \text{and}\quad \widetilde{C}=\frac{9R^4}{4}\,.
\end{equation}
\par
Notice that by (\ref{eq-betaygamma}), the mapping $r\rightarrow\gamma_r$ is increasing
for $r\in(0,1)$, and that $|\beta_r|=1/\cos(R\gamma_r)$.
\par
Let us rename $\Psi(r)=\Phi^2(r)$.  To prove \eqref{eq-sup2}, we distinguish among the following three cases.
\par\medskip
{\bf (i)}$\, \mathbf{\boldsymbol{\lambda}= 3/2.}$ In this case, by Theorem~\ref{prop-dilat}, $R=1$. We also have $a=1$, which gives $A=B=0$, $C=-3/2$ and therefore, $\widetilde{A}=\widetilde{B}=0$, $\widetilde{C}=9/4$. Hence
\[
\Psi(r)= \left(\widetilde{A}+\widetilde{B}|\b_r|^2+\widetilde{C}|\b_r|^4\right)\cdot\cos^4\gamma_r=\frac 94\cdot \frac{1}{\cos^4\gamma_r}\cdot\cos^4\gamma_r=\frac 94=\lambda^2\,.
\]
This proves that \eqref{eq-sup2} holds for $\lambda=3/2$.
\par\medskip
{\bf (ii)}$\, \mathbf{\boldsymbol{\lambda}> 3/2.}$  By Theorem~\ref{prop-dilat}, we have $R=1$. A straightforward calculation shows that $a>1$. Hence, using \eqref{eq-Atilde} and that $a+R>1$, we obtain $\widetilde{A}>0$. Now, since $|\beta_r|=1/\cos\gamma_r$, we can write
\begin{align*}
\nonumber\Psi(r)&=\left(\widetilde{A}+\widetilde{B}|\b|^2+\widetilde{C}|\b|^4\right)\cdot \cos^4\gamma_r
= \left(\frac{\widetilde{A}}{|\b|^4}+\frac{\widetilde{B}}{|\b|^2}+\widetilde{C}\right)\\
& =\widetilde{A}\cos^4(\gamma_r)+\widetilde{B}\cos^2(\gamma_r)+\widetilde{C}\,.
\end{align*}
Consider the function $\psi(x)=\widetilde{A}x^2+\widetilde{B}x+\widetilde{C},\ x\in (0,1)$. Note that $\Psi(r)=\psi(\cos^2(\gamma_r))$ and that $\Psi(r)\leq \Psi(0)$ if and only if $\psi(x)\leq \psi(1)$. To show that $\psi(x)\leq \psi(1)$ we argue as follows: the graph of $\psi$ is a convex parabola and $\psi(1)=\lambda^2> 9/4=\psi(0)$ since $\lambda > 3/2$. Hence, the unique critical point $x_0$ of $\psi$ is a minimum and satisfies $x_0< 1$. We have the following possibilities.
\begin{itemize}
\item[(i)] The critical point $x_0\leq 0$, which implies that for all $r\in [0,1)$, $\psi(r)\leq \psi(1)=\lambda^2$.
\item[(ii)] The critical point $x_0$ which is a minimum of $\psi$ belongs to $(0,1)$. In this case, we have
\[
\sup_{0\leq x\leq 1}\psi(x)=\max\{\psi(0), \psi(1)\}=\max\left\{\frac 94, \lambda^2\right\}=\lambda^2=\psi(1)\,,
\]
which shows that  \eqref{eq-sup2} holds also for $\lambda>3/2$.
\end{itemize}
\par\medskip
{\bf (iii)}$\, \mathbf{\boldsymbol{\lambda}< 3/2.}$  This is the most complicated case to analyze by far due to the amount of parameters that we are to control. Note that we have that $R<1$ by Theorem~\ref{prop-dilat}, and that $a<1$ by \eqref{eq-upperR}. Hence, using \eqref{eq-Atilde} and \eqref{eq-BCtilde}, we have that $\widetilde{A}<0$, while $\widetilde{B},\ \widetilde{C}>0$. Recall that $\beta_r=1+i\tan(R\gamma_r)$ and that $\cos\gamma_r=(1-r^2)/(1+r^2)$, which gives that $|\beta_r|=1/\cos(R\gamma_r)$ and that the correspondence $r\to\gamma_r$ is an (strictly) increasing function in $r\in (0,1)$ (and hence $\gamma_r'=\partial\gamma_r/\partial r>0$ for all such $r$).
\par
We are to show that $\sup_{0\leq r < 1} \Psi(r)=\Psi(0)$. To do so, we will check that the derivative of $\Psi$ is non-positive for all such $r$. For the convenience of the reader, we proceed in different steps.
\par
\par\smallskip
\textbf{Step 1: There exists $r_1\in (0,1)$ such that $\Psi'(r)< 0$ for all $r\in(0,r_1)$.}
As it was mentioned before, we can write
\[
\Psi(r)=\left(\widetilde{A}+\widetilde{B}|\b_r|^2+\widetilde{C}|\b_r|^4\right)\cdot \cos^4\gamma_r\,.
\]
Using this expression for $\Psi$, we compute its derivative to obtain
\begin{align*}
\nonumber\Psi'(r)&=\left(2\widetilde{B}|\b_r|+4\widetilde{C}|\b_r|^3\right)\cdot \frac{\sin(R\gamma_r)}{\cos^2(R\gamma_r)}\cdot R\gamma_r'\cdot \cos^4\gamma_r\\
& - \left(\widetilde{A}+\widetilde{B}|\b_r|^2+\widetilde{C}|\b_r|^4\right)\cdot 4\cos^3\gamma_r\cdot\sin\gamma_r\cdot \gamma_r'\,,
\end{align*}
which can be written as
\begin{align}\label{eq-derivative}
\Psi'(r)&=-\gamma_r'\cdot \frac{\cos^4\gamma_r\cdot\tan\gamma_r}{\cos(R\gamma_r)}\\
&\nonumber \times \left(4\widetilde{A}\cos(R\gamma_r)+2\widetilde{B}|\beta_r|+2\widetilde{B}\varphi_r|\beta_r|+4\widetilde{C}\varphi_r |\beta_r|^3\right)\,,
\end{align}
where
\[
\varphi_r=1-\frac{R\tan(R\gamma_r)}{\tan\gamma_r}\,.
\]
Let us see that $\varphi_r$ is an (strictly) increasing function of $r\in (0, 1)$. To do so, we compute the derivative of this function which is
\[
\varphi_r'=-\frac{R\gamma_r'}{\tan^2\gamma_r\cdot\cos\gamma_r\cdot\cos(R\gamma_r)}\cdot\left(\frac{R\sin\gamma_r}{\cos(R\gamma_r)}-\frac{\sin(R\gamma_r)}{\cos\gamma_r}\right).
\]
Now, $\varphi_r$ is increasing if and only if
\begin{equation}\label{eq-mu}
\frac{R\sin\gamma_r}{\cos(R\gamma_r)}<\frac{\sin(R\gamma_r)}{\cos\gamma_r}\quad \Leftrightarrow\quad R\sin(2\gamma_r)-\sin(2R\gamma_r)<0\,.
\end{equation}
Define $\mu(r)=R\sin(2\gamma_r)-\sin(2R\gamma_r),\ r\in(0,1)$. Then $\mu(0)=0$ and $$\mu'(r)=2R(\cos(2\gamma_r)-\cos(2R\gamma_r))\cdot\gamma_r'<0$$ since $0<R<1$, which gives \eqref{eq-mu} and shows that $\varphi_r$ is increasing for all $r\in(0,1)$.
\par
It is also clear that $\cos(R\gamma_r)$ is decreasing in $r$ so that $1/\cos(R\gamma_r)=|\beta_r|$ is increasing as well. Hence, the expression in the parentheses in \eqref{eq-derivative} is an increasing function of $r\in (0,1)$ (recall that $\widetilde{A}<0$ and that $\widetilde{B},\ \widetilde{C}>0$). On the other hand, $\cos^4\gamma_r\cdot\tan\gamma_r$ increases with $r$ for all $0<r<r_1$, where $\sin(\gamma_{r_1})=1/2$. This shows that as long as $0<r<r_1$, the function
\[
\widetilde{\varphi}(r)=-\frac{\cos^4\gamma_r\cdot\tan\gamma_r}{\cos(R\gamma_r)}
\cdot \left(4\widetilde{A}\cos(R\gamma_r)+2\widetilde{B}|\beta_r|+2\widetilde{B}\varphi_r|\beta_r|+4\widetilde{C}\varphi_r |\beta_r|^3\right)
\]
is decreasing. Since $\widetilde{\varphi}(0)=0$, we get that $\widetilde{\varphi}(r)< 0$ for all $r\in(0, r_1)$. In other words, keeping in mind that $\gamma'_r>0$ for all $r\in(0,1)$, we conclude from \eqref{eq-derivative} that $\Psi'(r)< 0$ for $r\in(0, r_1)$. (This implies, in particular, that $r=0$ is a local maximum of $\Psi$.)
\par
\par\smallskip
\textbf{Step 2: There exists $r_2\in (0,1)$ such that $\Psi'(r)< 0$ for all $r\in(r_2,1)$.} Recall that
\[
\lambda=2 a^2 + 2a R -2-R^2/2\,.
\]
Using \eqref{eq-p1} and \eqref{eq-betaygamma}, we get
\begin{align*}
\Psi(r)& =\left|2(1-a^2)+2R(R-a)\b_r-\frac{3R^2\b_r^2}{2}\right|^2\cdot\left(\frac{1-r^2}{1+r^2}\right)^4\\
&= \left({\rm Re}\left\{2(1-a^2)+2R(R-a)\b_r-\frac{3R^2\b_r^2}{2}\right\}\right)^2\cdot \cos^4\gamma_r\\
& +\left({\rm Im}\left\{2(1-a^2)+2R(R-a)\b_r-\frac{3R^2\b_r^2}{2}\right\}\right)^2\cdot \cos^4\gamma_r\\
&=\left(-\lambda+\frac{3R^2}{2}\cdot \tan^2(R\gamma_r)\right)^2\cdot \cos^4\gamma_r \\
& +\left(R^2+2aR\right)^2\cdot\tan^2(R\gamma_r) \cdot \cos^4\gamma_r\end{align*}
\begin{equation}\label{nueva}\hspace{1cm} = \lambda^2\cdot\cos^4\gamma_r+K\cdot\tan^2(R\gamma_r)\cdot\cos^4\gamma_r+\frac{9R^4}{4}\cdot\tan^4(R\gamma_r)\cdot\cos^4\gamma_r\,,
\end{equation}
where
\begin{eqnarray*}
K&=&R^4+4a^2R^2+4aR^3-3R^2\lambda=\frac{5R^4}{2}+6R^2-2a^2R^2-2aR^3\\
&\geq& \frac{5R^4}{2}+2R^2> 0\,,
\end{eqnarray*}
since $a$ and $R$ are less than one. Using \eqref{nueva} we see that the derivative of $\Psi$ with respect to $r$ equals
\begin{align*}
\Psi'(r)& =-4\lambda^2\cdot\cos^3\gamma_r\cdot\sin\gamma_r\cdot\gamma_r'+K\cdot\frac{2R\tan(R\gamma_r)}{\cos^2(R\gamma_r)}\cdot\cos^4\gamma_r\cdot\gamma_r'\\
& -4K\cdot\cos^3\gamma_r\cdot \sin\gamma_r\cdot\tan^2(R\gamma_r)\cdot\gamma_r'+9R^5\cdot \frac{\tan^3(R\gamma_r)}{\cos^2(R\gamma_r)}\cdot\cos^4\gamma_r\cdot\gamma_r'\\
& -9R^4\tan^4(R\gamma_r) \cdot\cos^3\gamma_r\cdot\sin\gamma_r\cdot\gamma_r'\,.
\end{align*}
Note that $\Psi'(r)<0$ if and only if
\begin{align}\label{eq-derivativesign}
\nonumber & K\cdot\frac{2R\tan(R\gamma_r)}{\cos^2(R\gamma_r)}\cdot\cos\gamma_r+9R^5\cdot \frac{\tan^3(R\gamma_r)}{\cos^2(R\gamma_r)}\cdot\cos\gamma_r\\
& < 4\lambda^2\cdot\sin\gamma_r+4K\cdot \sin\gamma_r\cdot\tan^2(R\gamma_r)+9R^4\tan^4(R\gamma_r) \cdot\sin\gamma_r\,,
\end{align}
an inequality the holds for all $r\geq r_2$, say, since as $r\to 1$ (or, equivalently, as $\gamma_r\to \pi/2$), the right hand side in \eqref{eq-derivativesign} goes to $0$ while the left hand side tend to a strictly positive real number. In other words, this shows that $\Psi'(r)<0$ for all $r\in (r_2,1)$.
\par
\par\smallskip
\textbf{Step 3: $\Psi$ has at most one critical point in $(0,1)$.} Note that once we check that the number of solutions of the equation $\Psi'(r)=0$, $r\in (0,1)$ is at most one, we will have, as a consequence of the previous steps that $\Psi'(r)\leq 0$ for $r\in(0,1)$ (and hence $\Psi$ is non-increasing in that interval). Observe that $\Psi'=0$ if and only if the function in the parentheses in \eqref{eq-derivative} is equal to zero. As was justified in Step 1, this function is increasing, which proves our claim.
\end{pf}
\par\smallskip
We summarize the previous analysis in the following proposition.
\begin{prop}\label{prop-extremal}
If $\lambda>0$ then the function $f_0\in\mathcal F_\lambda^0$.
\end{prop}

This shows that \eqref{eq-ineq1} and \eqref{eq-ineq2} are equalities, and finishes the proof of Theorem 1.


\begin{thebibliography}{10}


\bibitem{CH-D-O}
M. Chuaqui, P. L. Duren, and B. Osgood, The Schwarzian derivative for harmonic mappings, \emph{J. Anal. Math.} \textbf{91} (2003), 329--351.

\bibitem{CH-O}
M. Chuaqui and B. Osgood,  Sharp distorsion theorems associated with the Schwarzian derivative, \emph{J. London
Math. Soc.} \textbf{48} (1993), 289--298.


\bibitem{CSS}
J. Clunie and T. Sheil-Small, Harmonic univalent functions, \emph{Ann. Acad. Sci. Fenn. Ser. A} \textbf{9} (1984), 3--25.

\bibitem{Dur-Univ}
P. L. Duren, \emph{Univalent Functions}, Springer-Verlag, New York, 1983.

\bibitem{Dur-Harm}
P. L. Duren, \emph{Harmonic Mappings in the Plane}, Cambridge University
Press, Cambridge, 2004.

\bibitem{Eps}
C. Epstein, The hyperbolic Gauss map and quasiconformal reflections, \emph{J. Reine Angew. Math.} \textbf{372} (1986),
96--135.

\bibitem{Heins}
M. Heins, Some characterizations of finite Blaschke products of
positive degree, \textit{J. Anal. Math.\/} \textbf{46} (1986),
162--166.

\bibitem{HM1}
R. Hern\'andez and M. J. Mart\'{\i}n, Stable geometric pro\-per\-ties of analytic and harmonic functions, \emph{Math. Proc. Cambridge Philos. Soc.} \textbf{155} (2013), 343--359.

\bibitem{HM3}
R. Hern\'andez and M. J. Mart\'{\i}n, Quasiconformal extensions of harmonic mappings, \emph{Ann. Acad. Sci. Fenn. Ser. A. I Math.} \textbf{38} (2013), 617--630.

\bibitem{HM2}
R. Hern\'andez and M. J. Mart\'{\i}n, Pre-Schwarzian and Schwarzian de\-ri\-va\-ti\-ves of harmonic mappings, \emph{J. Geom. Anal.} DOI 10.1007/s12220-013-9413-x. Published electronically on April 13th, 2013.

\bibitem{HO}
T. Hosokawa and S. Ohno, Topological structures of the set of composition operators on
the Bloch space, \emph{J. Math. Anal. Appl.} \textbf{314} (2006), 736--748.


\bibitem{Kraus}
W. Kraus,  Uber den Zusammenhang einiger Charakteristiken eines einfach zusammenh\"angenden Bereiches mit der
Kreisabbildung, \emph{Mitt. Math. Sem. Geissen} \textbf{21} (1932), 1--28.

\bibitem{L}
H. Lewy, On the non-vanishing of the Jacobian in certain one-to-one mappings, \emph{Bull. Amer. Math. Soc.} \textbf{42} (1936), 689--692.

\bibitem{M}
M. J. Mart\'{\i}n, Hyperbolic distortion of conformal maps at corners, \emph{Constr. Approx.} \textbf{30} (2009), 265--275.
\bibitem{N}
Z. Nehari, The Schwarzian derivative and schlicht functions, \emph{Bull. Amer. Math. Soc.} \textbf{55} (1949),
545--551.

\bibitem{N2}
Z. Nehari, A property of convex conformal maps, \emph{J. Anal. Math.} \textbf{30} (1976), 390--393.

\bibitem{POM-I}
Ch. Pommerenke, Linear-invariante Familien analytischer Funktionen {I}, \emph{Math. Ann.} \textbf{155} (1964), 108--154.

\bibitem{POM-II}
Ch. Pommerenke, Linear-invarinte Familien analytischer Funktionen {II}, \emph{Math. Ann.} \textbf{156} (1964), 226--262.

\bibitem{P}
Ch. Pommerenke, \emph{Univalent Functions}, Vandenhoeck $\&$ Ruprecht, G\"{o}ttingen, 1975.

\bibitem{S-S}
T. Sheil-Small, Constants for planar harmonic mappings, \emph{J. London Math. Soc.} \textbf{42} (1990), 237--248.

\end{thebibliography}
\end{document}